\documentclass{article} \UseRawInputEncoding
\usepackage{amsmath,amsthm,amssymb,amsfonts}
\usepackage{verbatim}
\usepackage{graphicx}
\usepackage{stmaryrd} 
\usepackage{hyperref}

\usepackage[colorinlistoftodos]{todonotes}

\theoremstyle{plain}
\newtheorem{thm}{Theorem}

\newtheorem{prop}[thm]{Proposition}
\newtheorem{cor}[thm]{Corollary}

\theoremstyle{definition}
\newtheorem{definition}[thm]{Definition}

\newtheorem{assert}[thm]{Assertion}

\numberwithin{thm}{section}

\newcommand{\adj}{\leftrightarrow}
\newcommand{\adjeq}{\leftrightarroweq}

\def\Z{{\mathbb Z}}
\def\N{{\mathbb N}}

\title{Corrigendum for Han's Corrigendum}
\author{Laurence Boxer
\thanks{
    Department of Computer and Information Sciences,
    Niagara University,
    Niagara University, NY 14109, USA; \newline
    and \newline
    Department of Computer Science and Engineering,
    State University of New York at Buffalo. \newline
    email: boxer@niagara.edu
}
}

\date{ }
\begin{document}
\maketitle{}

\begin{abstract}
S.-E. Han's paper, "Remarks on Pseudocovering Spaces in a Digital Topological Setting:
A Corrigendum", is meant to address errors in previous papers. However, this paper
is also marked by errors in its mathematics, as well as improprieties in its citations.
We address these flaws in the current work.


\end{abstract}

\maketitle

\section{Introduction}
S.-E. Han's paper \cite{HanCorrig} claims in its title to be a Corrigendum
for his papers~\cite{HanUnique,HanEquiv}, but is
itself in need of a Corrigendum. These papers are concerned with
attempted variants on covering spaces in digital topology, among which are those that
turn out to coincide with digital covering spaces.
The current paper discusses flaws in Han's Corrigendum.

\begin{quote}
    The notion of a covering map has been adapted from
classical algebraic topology to digital topology,
where it is an important tool for computing
digital versions of fundamental groups for binary
digital images. \cite{BxVariants}
\end{quote}

We show that Han's ``Corrigendum" remains flawed in multiple ways:
\begin{itemize}
    \item One of the ``proofs" in~\cite{HanCorrig} is incorrect, although its
          assertion is correct; we discuss the error and give a correct proof 
          of the assertion.
    \item Much of~\cite{HanCorrig} is unoriginal, with original versions not credited.
    \item Han doesn't even cite his own papers correctly.     
\end{itemize}

\section{Preliminaries}
\subsection{Digital images}
Material in this section is largely quoted or paraphrased from~\cite{BxVariants}.

We use $\N$ for the set of natural numbers, 
$\Z$ for the set of integers, and
$\#X$ for the number of distinct members of $X$. Let $\N^* = \N \cup \{0\}$.

We typically denote a (binary) digital image
as $(X,\kappa)$, where $\emptyset \neq X \subset \Z^n$ for some
$n \in \N$ and $\kappa$ represents an adjacency
relation of pairs of points in $X$. Thus,
$(X,\kappa)$ is a graph, in which members of $X$ may be
thought of as black points, and members of $\Z^n \setminus X$
as white points, typically of a picture of some ``real world" 
object or scene.

Let $u,n \in \N$, $1 \le u \le n$. 
Han's papers use ``$k$-adjacency" sometimes to mean
an arbitrary adjacency, sometimes as an abbreviation
for what he calls ``$k(u,n)$-adjacency," where the
digital image $(X,k)$ satisfies $X \subset \Z^n$ and
$x = (x_1, \ldots, x_n),~y=(y_1, \ldots, y_n) \in X$ 
are $k(u,n)$-adjacent if and only if
\begin{itemize}
    \item $x \neq y$, and
    \item for at most $u$ indices~$i$, 
          $\mid x_i - y_i \mid = 1$, and
    \item for all indices $j$ such that 
          $\mid x_j - y_j \mid \neq 1$, we have
          $x_j = y_j$.
\end{itemize}
Generally, $n$ is either explicitly given (e.g., $n=2$) or is understood
as an abstract positive integer.
Therefore, other authors refer to this adjacency as $c_u$-adjacency.
We will prefer the latter notation in the current paper.
The $c_u$ adjacencies are the adjacencies most used
in digital topology, especially $c_1$ and $c_n$.

In low dimensions, it is also common to denote a
$c_u$ adjacency by the number of points that can
have this adjacency with a given point in $\Z^n$. E.g.,
\begin{itemize}
    \item For subsets of $\Z^1$, $c_1$-adjacency is 2-adjacency.
    \item For subsets of $\Z^2$, $c_1$-adjacency is 4-adjacency and
          $c_2$-adjacency is 8-adjacency.
    \item For subsets of $\Z^3$, $c_1$-adjacency is 6-adjacency,
          $c_2$-adjacency is 18-adjacency, and
          $c_3$-adjacency is 26-adjacency.
\end{itemize}

We use the notations $y \adj_{\kappa} x$, or, when
the adjacency $\kappa$ can be assumed, $y \adj x$, to mean
$x$ and $y$ are $\kappa$-adjacent.
The notations $y \adjeq_{\kappa} x$, or, when
$\kappa$ can be assumed, $y \adjeq x$, mean either
$y=x$ or $y \adj_{\kappa} x$.
For $x \in X$, let
\[ N(X,x,\kappa) =
  \{ \, y \in X \mid x \adj_{\kappa} y \, \}.
\]
When the image $(X,\kappa)$ under discussion is clear, we
will use the notations $N(x)$ or $N_{\kappa}(x)$
as follows.
\[
N(x) = \{ \, y \in X \mid y \adjeq_{\kappa} x \, \} =
N(X,x,\kappa) \cup \{x\}.
\]

A sequence $P=\{y_i\}_{i=0}^m$ in a digital image $(X,\kappa)$ is
a {\em $\kappa$-path from $a \in X$ to $b \in X$} if
$a=y_0$, $b=y_m$, and $y_i \adjeq_{\kappa} y_{i+1}$ 
for $0 \leq i < m$.

$X$ is {\em $\kappa$-connected}~\cite{Rosenfeld},
or {\em connected} when $\kappa$
is understood, if for every pair of points $a,b \in X$ there
exists a $\kappa$-path in $X$ from $a$ to $b$.

A {\em (digital) $\kappa$-closed curve} is a
path $S=\{s_i\}_{i=0}^m$ such that $s_0=s_m$,
and $0 < |i - j| < m$ 
implies $s_i \neq s_j$. If, also, $0 \le i < m$ implies
\[ N(S,x_i,\kappa)=\{x_{(i-1)\mod n},~x_{(i+1)\mod m}\}
\]
then $S$ is a {\em (digital) 
$\kappa$-simple closed curve}.

\subsection{Digitally continuous functions}
Material in this section is largely quoted or paraphrased from~\cite{BxVariants}.

Digital continuity is defined
to preserve connectedness, as at
Definition~\ref{continuous} below. By
using adjacency as our standard of ``closeness," we
get Theorem~\ref{continuityPreserveAdj} below.

\begin{definition}
\label{continuous}
{\rm ~\cite{Boxer99} (generalizing a definition of~\cite{Rosenfeld})}
Let $(X,\kappa)$ and $(Y,\lambda)$ be digital images.
A function $f: X \rightarrow Y$ is 
{\em $(\kappa,\lambda)$-continuous} if for
every $\kappa$-connected $A \subset X$ we have that
$f(A)$ is a $\lambda$-connected subset of $Y$.
\end{definition}

If either of $X$ or $Y$ is a subset of the 
other, we use the abbreviation
$\kappa$-continuous for $(\kappa,\kappa)$-continuous.

When the adjacency relations are understood, we will simply say that $f$ is \emph{continuous}. Continuity can be expressed in terms of adjacency of points:

\begin{thm}
{\rm ~\cite{Rosenfeld,Boxer99}}
\label{continuityPreserveAdj}
A function $f:X\to Y$ is continuous if and only if $x \adj x'$ in $X$ 
implies $f(x) \adjeq f(x')$.
\end{thm}

Han's papers generally use the equivalent formulation that
$f$ is continuous if and only for every $x \in X$,
$f(N_{\kappa}(x)) \subset N_{\lambda}(f(x))$.

See also~\cite{Chen94,Chen04}, where similar notions are referred to as {\em immersions}, {\em gradually varied operators},
and {\em gradually varied mappings}.

A digital {\em isomorphism} (called {\em homeomorphism}
in~\cite{Boxer94}) is a $(\kappa,\lambda)$-continuous
surjection $f: X \to Y$ such that $f^{-1}: Y \to X$ is
$(\lambda,\kappa)$-continuous.

The literature uses {\em path} polymorphically: a
$(c_1,\kappa)$-continuous function $f: [0,m]_{\Z} \to X$
is a $\kappa$-path if $f([0,m]_{\Z})$ is a $\kappa$-path
as described above from $f(0)$ to $f(m)$.

\subsection{Han's variants on local isomorphisms}
Material in this section is largely quoted or paraphrased from~\cite{BxVariants}.

The following Definition~\ref{coveringDef} of a digital covering map is
formally simpler but equivalent to the definition Han introduced in~\cite{Han05}.
The equivalence was shown in~\cite{BxAndWedges} and at about the same time
in~\cite{HanLectNotes}; it may have been derived independently by the authors of
these two papers. However, in~\cite{HanCorrig}, Definition~\ref{coveringDef} is attributed to three of Han's papers without
acknowledging its origin in~\cite{BxAndWedges,HanLectNotes}.

\begin{definition}
   {\rm \cite{BxAndWedges}}
   \label{coveringDef}
   Let $p: (E,\kappa) \to (B,\lambda)$ be a continuous
   surjection of digital images. The map $p$ is a
   {\em $(\kappa,\lambda)$-covering (map)} if and only if
   \begin{enumerate}
       \item for every $b \in B$, there is an index set~$M$
             such that 
             \[p^{-1}(N_{\lambda}(b)) = 
             \bigcup_{i \in M} N_{\kappa}(e_i), 
             \mbox{ where } e_i \in p^{-1}(b);
             \]
       \item if $i,j \in M$, $i \neq j$, then
            $N_{\kappa}(e_i) \cap N_{\kappa}(e_j) =
            \emptyset$; and
       \item $p|_{N_{\kappa}(e_i)}: N_{\kappa}(e_i) \to
              N_{\lambda}(b)$ is a 
              $(\kappa,\lambda)$-isomorphism.
   \end{enumerate}
\end{definition}

We find the following definition in Han's
paper~\cite{Han20} (not in~\cite{Han04} despite the claims
to the contrary in~\cite{Han20,HanEquiv}).

\begin{definition}
\label{PL-iso}
A digitally continuous map $h: (X,\kappa) \to (Y,\lambda)$
is a {\em pseudo-local (PL) isomorphism} if for every
$x \in X$, $h(N_{\kappa}(x)) \subset Y$ is
$\lambda$-isomorphic to 
$N_{\lambda}(h(x)) \subset Y$.
\end{definition}

In his paper~\cite{Han04}, Han gives the following.

\begin{definition}
\label{localIso}
A digitally continuous map $h: (X,\kappa) \to (Y,\lambda)$
is a {\em local homeomorphism} [in more recent
terminology, a {\em local isomorphism}]
if for all $x \in X$, $h|_{N_{\kappa}(x)}$ is a 
$(\kappa,\lambda)$-homeomorphism
[$(\kappa,\lambda)$-isomorphism] onto
$N_{\lambda}(h(x))$.
\end{definition}

The following was noted in~\cite{BxVariants}.

\begin{prop}
\label{localPLequiv}
Let $h: (X,\kappa) \to (Y,\lambda)$ be a digitally continuous map.
If $h$ is a local isomorphism then $h$ is a PL isomorphism.
\end{prop}

\begin{thm}
    {\rm (\cite{PakZak}, correcting an error of~\cite{Han04})}
    \label{PakZakEquiv}
    Let $f: (X,\kappa) \to (Y,\lambda)$ be a continuous
    surjection. Then $f$ is a digital covering map
    if and only if $f$ is a local isomorphism.
\end{thm}

We will also discuss the following notion.

\begin{definition}
    {\rm \cite{HanUnique}}
    \label{WL-iso}
    A function $h: (X,\kappa) \to (Y,\lambda)$
    is a {\em weakly local (WL) isomorphism} if for all
    $x \in X$, $h|_{N(x,1)}$ is an isomorphism onto
    $h(N(x,1))$.
\end{definition}

\section{Han's pseudo-covering maps in~~\cite{HanEquiv}}
Han defines a digital pseudo-covering as follows.

\begin{definition}
\label{HanPseudocover}
{\rm \cite{HanUnique}}
    Let $p: (E,\kappa) \to (B,\lambda)$ be a surjection
    such that for every $b \in B$,
    \begin{enumerate}
        \item there is an index set $M$ such that
              $p^{-1} (N_{\kappa}(b,1)) = 
              \bigcup_{i \in M} N_{\kappa}(e_i,1)$, where
              $e_i \in p^{-1}(b)$;
        \item if $i,j \in M$ and $i \neq j$, then
              $N_{\kappa}(e_i,1) \cap N_{\kappa}(e_j,1) =
              \emptyset$; and
        \item $p|_{N_{\kappa}(e_i,1)}: N_{\kappa}(e_i,1) \to
               N_{\lambda}(b,1)$ is a WL-isomorphism for 
               all $i \in M$.
    \end{enumerate}
    Then $p$ is a {\em pseudo-covering} map.
\end{definition}

However, A. Pakdaman shows in~\cite{Pak22} that Han's definition
does not effectively give us a new object of study.
In particular, Pakdaman shows the following.

\begin{thm}
{\rm \cite{Pak22}}
\label{pseudoCoverIsCover}
    A digital pseudo-covering map as defined in
    Definition~\ref{HanPseudocover} is in fact a digital covering map.
\end{thm}

\begin{definition}
    {\rm \cite{Han05}} Let $p: (E,\kappa) \to (B,\lambda)$
    be $(\kappa,\lambda)$-continuous. Let 
    $f: [0,m]_{\Z} \to B$ be $(c_1, \lambda)$-continuous.
    A $(c_1,\kappa)$-continuous function 
    $\tilde{f}: [0,m]_{\Z} \to E$ such that
    $p \circ \tilde{f} = f$ is a 
    {\em (digital) path lifting of $f$}. If for every
    $b_0 \in B$, every $e_0 \in p^{-1}(b_0)$, and every
    path $f$ such that $f(0)=b_0$, there is a unique
    lifting $\tilde{f}$ such that $\tilde{f}(0)=e_0$,
    then $p$ has the {\em unique path lifting property}.
\end{definition}

Han's variants of covering maps don't all give us new notions, as
shown by the following expansion of Theorem~\ref{pseudoCoverIsCover}.

\begin{thm}
{\rm \cite{BxVariants}}
\label{coverEquivs}
 Let $p: (X,\kappa) \to (Y,\lambda)$ be a continuous
    surjection. Then the following are equivalent.
    \begin{itemize}
        \item $p$ is a digital covering map.
        \item $p$ is a local isomorphism.
        \item $p$ is a pseudo-covering in the sense of
            Definition~\ref{HanPseudocover}.
        \item $p$ is a WL-isomorphism with the unique
              path lifting property.
\end{itemize}
\end{thm}

In order to obtain a notion of a digital pseudo-cover that is not equivalent to
a digital cover, Pakdaman gave the following definition.

\begin{definition}
    \label{PakPseudocovering}
    {\rm \cite{Pak22}}
    Let $p: (E,\kappa) \to (B,\lambda)$ be a surjection
    of digital images. Suppose for all $b \in B$ we
    have the following.
    \begin{enumerate}
        \item for some index set $M$,
              $\bigcup_{i \in M} N_{\kappa}(e_i,1)
              \subset p^{-1}(N_{\lambda}(b,1))$ where
              $e_i \in p^{-1}(b)$;
        \item if $i,j \in M$ and $i \neq j$ then
              $N_{\kappa}(e_i,1) \cap N_{\kappa}(e_i,1) =
              \emptyset$; and
        \item for all $i \in M$,
              $p|{N_{\kappa}(e_i,1)}: N_{\kappa}(e_i,1) \to
               p(N_{\kappa}(e_i,1))$ is a
               $(\kappa, \lambda)$-isomorphism.
    \end{enumerate}
    Then $p$ is a {\em $(\kappa,\lambda)$-pseudocovering map}.
\end{definition}

Pakdaman~\cite{Pak22} showed that Definition~\ref{PakPseudocovering}, unlike
Definition~\ref{HanPseudocover}, gives us a notion of a digital pseudo-cover
that is not a digital cover. He did this by showing that, unlike a digital
cover, the version of digital pseudo-cover in Definition~\ref{PakPseudocovering}
need not have the unique lifting property.

Having recognized that his previous definition of a digital pseudo-cover, stated here as Definition~\ref{HanPseudocover}, failed to achieve his goal, Han then stated
in~\cite{HanCorrig} the same alternate definition of a digital pseudo-cover 
that Pakdaman gave (above, Definition~\ref{PakPseudocovering}), without attributing it
to Pakdaman (see Han's Definition~4.1 in~\cite{HanCorrig}). Since~\cite{HanCorrig}
does acknowledge the existence of Pakdaman's~\cite{Pak22}, this serves as an
example of, at best, careless research by Han, and perhaps a breach of ethics.

\section{Other flaws in \cite{HanCorrig}}
In this section, we take $p: \N^* \to C_n$, where $C_n$ is a digital
simple closed curve of~$n$ points $\{c_0, c_1, \ldots, c_{n-1} \}$ indexed
circularly, and~$p$ is defined by
\[ p(z) = c_{z \! \! \mod n}.\]

In Remark 3.10(1) of~\cite{HanCorrig}, Han claims

\begin{assert}
    \label{Han3.10-1}
    There exists $b \in C_n$ such that
    \[ \bigcup_{i \in \N^*} N(e_i) \neq p^{-1}(N(b)), ~~~ \mbox{ where } ~~~
       p^{-1}(b) = \{ e_i \mid i \in \N^* \}.
    \]
\end{assert}

\begin{prop}
    \label{Han3.10False}
    Assertion~\ref{Han3.10-1} is false.
\end{prop}

\begin{proof}
Consider the following cases.
   \begin{itemize}
        \item Suppose $b=c_0$. We have 
             \[ p^{-1}(N(c_0)) = \{0,1 \} \cup \bigcup_{i=1}^{\infty} [i \cdot n -1, i \cdot n+1]_{\Z}
                = \bigcup_{i \in \N^*} N(e_i),
             \]
             where $e_i = i \cdot n$ and $p^{-1}(0) = \{ e_i \mid i \in \N^* \}$.
        \item Suppose $b=c_j$ such that $j > 0$. We have 
             \[ p^{-1}(N(c_j)) = \bigcup_{i=0}^{\infty} [j + i \cdot n -1, j + i \cdot n+1]_{\Z}
                = \bigcup_{i \in \N^*} N(e_i),
             \]
             where $e_i = j + i \cdot n$ and $p^{-1}(c_j) = \{ e_i \mid i \in \N^* \}$.
    \end{itemize}
    Therefore we have  $p^{-1}(N(b)) = \bigcup_{i \in \N^*} N(e_i)$ in both
    cases. This completes the proof.
\end{proof}

Han claims that the following Corollary~\ref{Han3.11} (stated
in\cite{HanCorrig} as Corollary~3.11) follows from his Remark~3.10.
However, since we see in Proposition~\ref{Han3.10False} that Han's
Remark~3.10 is false, Corollary~\ref{Han3.11} remains in need of a proof,
which we give below.

\begin{cor}
    \label{Han3.11}
    A pseudo-covering map in the sense of Definition~\ref{HanPseudocover} is
    a WL-isomorphic surjection, but the converse does not obtain.
\end{cor}

\begin{proof}
    By Theorem~\ref{coverEquivs}, a pseudo-covering map in the sense of 
    Definition~\ref{HanPseudocover} is a WL-isomorphic surjection.

    The map $p$ assumed for this section is easily seen to be a WL-isomorphic
    surjection. However, as in the proof of Proposition~3.2 
    of~\cite{Pak22},~$p$ lacks the unique path lifting property:
    Given the path $P=\{s_0, s_{n-1} \}$ in~$C_n$, there is no lift of~$P$
    to a path in~$N^*$ from $0 \in p^{-1}(s_0)$ to any member of~$p^{-1}(s_{n-1})$.
    It follows from Theorem~\ref{coverEquivs} that~$p$ is not a pseudo-covering map
    in the sense of Definition~\ref{HanPseudocover}.
\end{proof}

\section{Further remarks}
S.-E. Han's paper~\cite{HanCorrig} is shown in the current work to have a major
mathematical error, trivialities (having introduced variants
on the notion of a covering map that don't really vary), 
content that is unoriginal but presented as if original,
and citations that are incorrect or inappropriate.
These are features of many of Han's papers, examples of which, in addition
to those discussed above, include: \newline

\begin{tabular}{l l}
    Han~paper & Flaws~discussed~in \\ \hline
    \cite{HanDeform} & \cite{BxRemarksHanDeform} \\
     \cite{Han05} & \cite{BxAndWedges} \\
    \cite{HanPi1&Euler} & \cite{BxSt16} \\
    \cite{Kang-Han} & \cite{ColdAndFreez}
\end{tabular}
\newline \newline 

Perhaps the outstanding example of Han's unethical conduct is in his
paper~\cite{plagiarized}, which is difficult to find
as its vehicle, the {\em Journal of Computer $\&$ Communications Research},
is not an international journal; rather, this ``journal" is described by
Han as a local publication of
Honam University, at which Han was employed when the paper was written. This
paper is largely plagiarized from \cite{Boxer94,Boxer99} and has often been cited
by Han after he pledged not to do so -- documentation available
upon request.

Authors are responsible for their writings, but referees who approve publication
of mathematical trash also bear some responsibility. The point is not that a
Han submission should be automatically rejected - he 
has introduced useful ideas into digital topology,
such as covering maps and the 
shortest-path metric. Rather, his submissions should 
be reviewed with the greatest
care, by reviewers with good knowledge of the relevant literature, who are not afraid
to recommend either rejection or major revision.

\section{Acknowledgment}
We are grateful to the anonymous referees for their
suggestions and corrections.

\end{document}